\documentclass[psamsfonts]{amsart}
\usepackage{amsmath}
\usepackage{amssymb}
\usepackage{amsthm}
\usepackage{mathrsfs}
\usepackage[cp1252]{inputenc}
\usepackage{hyperref}
\usepackage{fancyhdr}
\theoremstyle{plain}
\newtheorem{thm}{Theorem}[section]
\newtheorem{pr}{Proposition}[section]
\newtheorem{co}{Corollary}[section]
\newtheorem{lem}{Lemma}[section]
\theoremstyle{remark}
\newtheorem*{rem}{Remark}
\newtheorem*{nota}{Notation}
\newtheorem*{notas}{Notations}
\theoremstyle{definition}
\newtheorem*{de}{Definition}
\numberwithin{equation}{section}
\begin{document}
\bibliographystyle{plain}
\title{Fefferman-Stein inequalities for the $\mathbb Z_2^d$ Dunkl maximal operator}
\author{Luc DELEAVAL}
\address{Institut de Mathématiques de Jussieu, Université Pierre et Marie Curie, 175, rue du Chevaleret, 75013 Paris, France}
\email{deleaval@math.jussieu.fr}
\thanks{The author is pleased to express his respectful thanks to the referee for his/her careful reading of the manuscript and for his/her comments which contributed to the improvement of the quality of the paper. He also wishes to thank his supervisor Sami Mustapha for sharing his ideas with him.}
\begin{abstract} 
In this article, we establish the Fefferman-Stein inequalities for the Dunkl maximal operator associated with a finite reflection group generated by the sign changes. Similar results are also given for a large class of operators related to Dunkl's analysis.
\end{abstract}
\maketitle
\section{Introduction}
In the early seventies, C. Fefferman and E. M. Stein have proved in \cite{fs} the following extension of the Hardy-Littlewood maximal theorem.
\begin{thm}
Let  $(f_n)_{n\geqslant1}$ be a sequence of measurable functions defined on $\mathbb R^d$ and let $M$ be the well-known maximal operator given by 
\[
Mf(x)=\sup \frac{1}{m(Q)}\int_Q|f(y)|\,\mathrm{d}y, \quad x \in \mathbb R^d,
\]
where the $\sup$ is taken over all cubes $Q$ centered at $x$ and $m(X)$ is the Lebesgue measure of $X$. 
\begin{enumerate}
\item If $1<r<+\infty$, $1<p<+\infty$ and if $\bigl(\sum_{n=1}^\infty|f_n(\cdot)|^r\bigr)^\frac{1}{r} \in L^p(\mathbb R^d;\mathrm{d}m)$, then we have 
\[
\biggl\|\Bigl(\sum_{n=1}^\infty|Mf_n(\cdot)|^r\Bigr)^\frac{1}{r}\biggr\|_p\leqslant C\biggl\|\Bigl(\sum_{n=1}^\infty|f_n(\cdot)|^r\Bigr)^\frac{1}{r}\biggr\|_p,
\]
where $C=C(r,p)$ is independent of $(f_n)_{n\geqslant1}$.
\item If $1<r<+\infty$ and if $\bigl(\sum_{n=1}^\infty|f_n(\cdot)|^r\bigr)^\frac{1}{r} \in L^1(\mathbb R^d;\mathrm{d}m)$, then  for every $\lambda>0$ we have
\[
m\biggl(\biggl\{x \in \mathbb R^d: \Bigl(\sum_{n=1}^\infty|Mf_n(x)|^r\Bigr)^\frac{1}{r}>\lambda\biggr\} \biggr)\leqslant \frac{C}{\lambda}\biggl\|\Bigl(\sum_{n=1}^\infty|f_n(\cdot)|^r\Bigr)^\frac{1}{r}\biggr\|_1,
\]
where $C=C(r)$ is independent of $(f_n)_{n\geqslant1}$ and $\lambda$.
\end{enumerate}
\end{thm}
One would like to extend this result to the case of the Dunkl maximal operator~$M_\kappa$ which is defined according to S. Thangavelu and Y. Xu (see \cite{tx}) by 
\[
M_\kappa f(x)=\sup_{r>0}\,\frac{1}{\mu_\kappa(B_r)}\,\bigl|(f*_\kappa\chi_{_{B_r}})(x)\bigr|, \quad x \in \mathbb R^d,
\]
where we denote by $\chi_{_{X}}$ the characteristic function of the set $X$, by $B_r$ the Euclidean ball centered at the origin and whose radius is $r$, by $\mu_\kappa$ a weighted Lebesgue measure invariant under the action of a finite reflection group and by $*_\kappa$ the Dunkl convolution operator (see Section $2$ for more details). 
\newline \indent However, the lack of information on this convolution, which is defined through a generalized translation operator (also called Dunkl translation), prevents from stating a general result. Just as in the  study of the weighted Riesz transform associated with the Dunkl transform (see \cite{tx2}), we can only establish a complete result for the finite reflection group $G\simeq\mathbb Z_2^d$ with the associated measure $\mu_\kappa$ given for every $x=(x_1,\ldots,x_d) \in \mathbb R^d$ by
\begin{equation} \label{measu}
\mathrm{d}\mu_\kappa(x)=h_\kappa^2(x)\,\mathrm{d}x,
\end{equation}
with $h_\kappa$ the $\mathbb Z_2^d$-invariant function defined by
\[
h_\kappa(x)=\prod_{j=1}^d|x_j|^{\kappa_j}=\prod_{j=1}^dh_{\kappa_j}(x_j),
\]
where $\kappa_1,\ldots,\kappa_d$ are nonnegative real numbers (let us note that $h_\kappa$ is homogeneous of degree $\gamma_\kappa=\sum_{j=1}^d\kappa_j$).
\newline \indent To become more precise, the aim of this paper is to prove the following Fefferman-Stein inequalities, where we denote by $L^p(\mu_\kappa)$ the space $L^p(\mathbb R^d;\mathrm d\mu_\kappa)$ and we use the shorter notation $\mathopen\|\cdot\mathclose\|_{\kappa,p}$ instead of $\mathopen\|\cdot\mathclose\|_{L^p(\mu_\kappa)}$. For $p \in [1,+\infty]$, the space~$L^p(\mu_\kappa)$ is of course the space of measurable functions on $\mathbb R^d$ such that 
\begin{alignat*}{4}
&\|f\|_{\kappa,p}&&=\biggl(\int_{\mathbb R^d}|f(y)|^p\,\mathrm{d}\mu_\kappa(y)\biggr)^\frac{1}{p}<+\infty\ \ \ \ &&\text{if}\ 1\leqslant p<+\infty, \\
&\|f\|_{\kappa,\infty}&&=\text{ess}\sup_{y \in \mathbb R^d}|f(y)|<+\infty\ \ \ \ &&\text{otherwise}.
\end{alignat*}
\begin{thm} \label{mainresult}
Let $G\simeq\mathbb Z_2^d$ and let $\mu_\kappa$ be the measure given by \eqref{measu}. Let  $(f_n)_{n\geqslant1}$ be a sequence of measurable functions defined on $\mathbb R^d$. 
\begin{enumerate} 
\item If $1<r<+\infty$, $1<p<+\infty$ and if $\bigl(\sum_{n=1}^\infty|f_n(\cdot)|^r\bigr)^\frac{1}{r} \in L^p(\mu_\kappa)$, then we have 
\[
\biggl\|\Bigl(\sum_{n=1}^\infty|M_\kappa f_n(\cdot)|^r\Bigr)^\frac{1}{r}\biggr\|_{\kappa,p}\leqslant C\biggl\|\Bigl(\sum_{n=1}^\infty|f_n(\cdot)|^r\Bigr)^\frac{1}{r}\biggr\|_{\kappa,p},
\]
where $C=C(\kappa_1,\ldots,\kappa_d,r,p)$ is independent of $(f_n)_{n\geqslant1}$.
\item If $1<r<+\infty$ and if  $\bigl(\sum_{n=1}^\infty|f_n(\cdot)|^r\bigr)^\frac{1}{r} \in L^1(\mu_\kappa)$, then for every $\lambda>0$ we have
\[
\mu_\kappa\biggl(\biggl\{x \in \mathbb R^d: \Bigl(\sum_{n=1}^\infty|M_\kappa f_n(x)|^r\Bigr)^\frac{1}{r}>\lambda\biggr\} \biggr)\leqslant \frac{C}{\lambda}\biggl\|\Bigl(\sum_{n=1}^\infty|f_n(\cdot)|^r\Bigr)^\frac{1}{r}\biggr\|_{\kappa,1},
\]
where $C=C(\kappa_1,\ldots,\kappa_d,r)$ is independent of $(f_n)_{n\geqslant1}$ and $\lambda$.
\end{enumerate}
\end{thm}
The proof of Theorem 1.1 is mainly based on a maximal theorem, a Calder\'on-Zygmund decomposition and a weighted inequality. Nevertheless, the Dunkl maximal operator cannot be treated by this method even if a maximal theorem has been established for this one in \cite{tx}. This is closely related to the fact that a theory of singular integrals associated with the Dunkl transform seems to be out of reach at the moment.
\newline
\indent In order to bypass this problem, we will construct a weighted maximal operator~$M^R_\kappa$ of Hardy-Littlewood type which satisfies the classical Fefferman-Stein inequalities and which controls $M_\kappa$ in the sense that for every $x \in \mathbb R^d_{\,\text{reg}}$
\begin{equation} \label{eq1}
M_\kappa f(x)\leqslant CM^R_\kappa f(x),
\end{equation}
where $C$ is a positive constant independent of $x$ and $f$ and where we set
\[
\mathbb R^d_{\,\text{reg}}=\mathbb R^d\setminus \bigcup_{j=1}^d\bigl\{x=(x_1,\ldots,x_d) \in \mathbb R^d: x_j=0\bigr\}.
\]
\indent The paper is organised as follows. 
\newline
In the next section, we collect some definitions and results related to Dunkl's analysis. In particular, we list the properties of the Dunkl transform (and the associated tools) which will be relevant for the sequel.
\newline
Section 3 is devoted to the proof of Theorem \ref{mainresult}. In view of this, we will prove the inequality \eqref{eq1} thanks to a more convenient Dunkl maximal operator $M_\kappa^Q$ and we will explain why the classical Fefferman-Stein inequalities hold for the operator~$M_\kappa^R$. Therefore, there will be nothing more to do to conclude that Theorem \ref{mainresult} is true.
\newline
An application of our Fefferman-Stein inequalities is given in Section 4.
\newline
\indent Throughout this paper, $C$ denotes a positive constant, which depends only on fixed parameters, and whose value may vary from line to line.
\section{Preliminaries}
This section is devoted to the preliminaries and background. These concern in particular the intertwining operator, the Dunkl transform, the Dunkl translation and the Dunkl convolution. We restrict the statement from Dunkl's analysis to the special case considered in this article. For a large survey about this theory, the reader may especially consult \cite{dJ, d2, ro5, ro4, tx, tr}.
\newline
\indent Let $e_1,\ldots,e_d$ be the standard basis of $\mathbb R^d$. We denote by $\sigma_j$ (for each $j$ from $1$ to $d$) the reflection with respect to the hyperplane perpendicular to $e_j$, that is to say for every $x=(x_1,\ldots,x_d) \in \mathbb R^d$
\[
\sigma_j(x)=x-2\tfrac{\left\langle x,e_j\right\rangle }{\|e_j\|^2}e_j=(x_1,\ldots,x_{j-1},-x_j,x_{j+1},\ldots,x_d).
\]
Of course $\mathopen\langle \cdot{,}\cdot \mathclose\rangle$ is the usual inner product on $\mathbb R^d\times\mathbb R^d$ and $\mathopen\|\cdot\mathclose\|$ is the associated norm. Let G be the finite reflection group generated by $\{\sigma_j: j=1,\ldots,d\}$, so $G$ is isomorphic to $\mathbb Z_2^d$. Let $\kappa_1,\kappa_2,\ldots,\kappa_d$ be nonnegative real numbers. \newline
\indent Associated with these objects are the Dunkl operators $\mathcal D_k$ (for $k=1,\ldots,d$) which have been introduced in \cite{d3} by C.~{F. Dunkl.} They are given for $x \in \mathbb R^d$ by
\[
\mathcal D_kf(x)=\partial_kf(x)+\sum_{j=1}^d\kappa_j\frac{f(x)-f\bigl(\sigma_j(x)\bigr)}{\left\langle x,e_j\right\rangle }\left\langle e_k,e_j\right\rangle=\partial_kf(x)+\kappa_k\frac{f(x)-f\bigl(\sigma_k(x)\bigr)}{x_k},
\]  
where $\partial_k$ denotes the usual partial derivative. A fundamental property of these differential-difference operators is their commutativity, that is to say $\mathcal D_k\mathcal D_l=\mathcal D_l\mathcal D_k$. 
\newline
\indent Closely related to them is the so-called intertwining operator $V_\kappa$ (the subscript means that the operator depends on the parameters $\kappa_j$, except in the rank-one case where the subscript is then a single parameter) which is the unique linear isomorphism of $\mathcal \bigoplus_{n\geqslant0}\mathcal P_n$ such that
\[
V_\kappa(\mathcal P_n)=\mathcal P_n,\quad V_\kappa(1)=1,\quad \mathcal D_kV_\kappa=V_\kappa\partial_k\ \, \text{for}\ \, k=1,\ldots,d,
\]
with $\mathcal P_n$ the subspace of homogeneous polynomials of degree $n$ in $d$ variables. Even if the positivity of the intertwining operator has been established in \cite{ro2} by M.~{Rösler}, an explicit formula of $V_\kappa$ is not known in general. However, in our setting, the operator $V_\kappa$ is given according to \cite{xuin} by the following integral representation
\[
V_\kappa f(x)=\int_{[-1,1]^d}f(x_1t_1,\ldots,x_dt_d)\prod_{j=1}^dM_{\kappa_j}(1+t_j)(1-t_j^2)^{\kappa_j-1}\,\mathrm{d}t,
\]
with $M_{\kappa_j}=\tfrac{\Gamma(\kappa_j+\frac{1}{2})}{\Gamma(\kappa_j)\Gamma(\frac{1}{2})}$ (where $\Gamma$ is the well-known Gamma function).
\newline
\indent In order to define the Dunkl transform, we also need to introduce the Dunkl kernel $E_\kappa$ which is given for $x \in \mathbb C^d$  by 
\[
E_\kappa(\cdot,x)(y)=V_\kappa\bigl(\mathrm e^{\left\langle \cdot,x\right\rangle }\bigr)(y), \quad y \in \mathbb R^d.
\]
It has a unique holomorphic extension to $\mathbb C^d\times\mathbb C^d$ and it satisfies the following basic properties: $E_\kappa(x,y)=E_\kappa(y,x)$ for $x, y \in \mathbb C^d$,  $E_\kappa(x,0)=1$ for $x \in \mathbb C^d$ and~$|E_\kappa(ix,y)|\leqslant 1$ for $x, y \in \mathbb R^d$. 
Considering the definition of $E_\kappa$ together with the explicit formula for $V_\kappa$ gives us
\[
E_\kappa(x,y)=\prod_{j=1}^dE_{\kappa_j}(x_j,y_j).
\]
In the rank-one case, $E_\kappa$ is explicitly known. More precisely, it is given for both $x$ and $y$ in $\mathbb C$ by
\[
E_\kappa(x,y)=j_{\kappa-\frac{1}{2}}(ixy)+\frac{xy}{2\kappa+1}j_{\kappa+\frac{1}{2}}(ixy),
\]
where $j_\kappa$ is the normalized Bessel function of the first kind and of order $\kappa$ (see~\cite{wa}). 
Moreover, we have a crucial one-dimensional product formula for this kernel. Before formulating it, let us introduce some notations.
\begin{notas}
\begin{enumerate} 
\item
For $x,y,z \in \mathbb R$, we put 
\[\sigma_{x,y,z}=
\begin{cases}
\frac{1}{2xy}(x^2+y^2-z^2)& \text{if } x,y \neq 0, \\
0& \text{if } x=0 \text{ or } y=0,
\end{cases}
\]
as well as
\[
\varrho(x,y,z)=\frac{1}{2}(1-\sigma_{x,y,z}+\sigma_{z,x,y}+\sigma_{z,y,x}).
\]
\item
For $x,y,z>0$, we put
\[
K_\kappa(x,y,z)=2^{2\kappa-2}M_\kappa\frac{\Delta(x,y,z)^{2\kappa-2}}{(xyz)^{2\kappa-1}}\chi_{{}_{[|x-y|,x+y]}}(z),
\]
where  $\Delta(x,y,z)$ denotes the area of the triangle (perhaps degenerated) with sides $x,y,z$.
\end{enumerate}
\end{notas}
With these notations in mind, we can now state the product formula for the Dunkl kernel (this formula has been proved in \cite{ro1} in the more general setting of signed hypergroups).
\begin{pr} \label{productformula}
Let $x, y \in \mathbb R$. 
\begin{enumerate}
\item
For every $\lambda \in \mathbb R$ we have
\[
E_{\kappa}(ix,\lambda)E_{\kappa}(iy,\lambda)=\int_{\mathbb R}E_{\kappa}(i\lambda,z)\,\mathrm{d}\nu^\kappa_{x,y}(z),
\]
where the measure $\nu^\kappa_{x,y}$ is given by
\[
\mathrm d\nu^\kappa_{x,y}(z)=
\begin{cases}
\mathcal K_\kappa(x,y,z)\,\mathrm{d}\mu_\kappa(z)& \text{if } x,y \neq 0, \\
\mathrm d\delta_x(z)& \text{if } y=0, \\
\mathrm d\delta_y(z)& \text{if } x=0,
\end{cases}
\]
with 
\[
\mathcal K_\kappa(x,y,z)=K_\kappa\bigl(|x|,|y|,|z|\bigr)\varrho(x,y,z).
\]
\item The measure $\nu^\kappa_{x,y}$ satisfies
\begin{enumerate}
\item $\mathrm{supp}\,\nu^\kappa_{x,y}=\Bigl[-|x|-|y|,-\bigl||x|-|y|\bigr|\Bigr]\bigcup\Bigl[\bigl||x|-|y|\bigr|,|x|+|y|\Bigr]$ for $x, y \neq 0$.
\item $\nu^\kappa_{x,y}(\mathbb R)=1$ and $\|\nu^\kappa_{x,y}\|\leqslant 4$, for $x, y \in \mathbb R$.
\end{enumerate}
\end{enumerate}
\end{pr}
We are now in a position to introduce the Dunkl transform which is taken with respect to the measure $\mu_\kappa$ defined by \eqref{measu}. For $f \in L^1(\mu_\kappa)$, the Dunkl transform of $f$, denoted by $\mathcal F_\kappa(f)$, is given by
\[
\mathcal F_\kappa(f)(x)=c_\kappa\int_{\mathbb R^d}f(y)E_\kappa(x,-iy)\,\mathrm d\mu_\kappa(y), \quad x \in \mathbb R^d,
\]
where $c_\kappa$ is the following  constant
\[
c_\kappa^{-1}=\int_{\mathbb R^d}\mathrm{e}^{-\frac{\|x\|^2}{2}}\,\mathrm{d}\mu_\kappa(x)=\prod_{j=1}^dc_{\kappa_j}^{-1}.
\] 
If $\kappa_	1=\cdots=\kappa_d=0$, then $V_\kappa=\text{id}$ and the Dunkl transform coincides with the Euclidean Fourier transform. In the rank-one case,  it is more or less a Hankel transform (see~\cite{wa}). The following proposition (see \cite{dJ}) gives us a Plancherel theorem and an inversion formula.
\begin{pr} \label{plancherel}
\begin{enumerate}
\item The Dunkl transform extends uniquely to an isometric isomorphism of $L^2(\mu_\kappa)$.
\item If both $f$ and $\mathcal F_\kappa(f)$ are in $L^1(\mu_\kappa)$ then 
\[
f(x)=c_\kappa\int_{\mathbb R^d}\mathcal F_\kappa(f)(y)E_\kappa(ix,y)\,\mathrm d\mu_\kappa(y).
\]
\end{enumerate}
\end{pr}
The Dunkl transform shares many other properties with the Fourier transform. Therefore, it is natural to associate a generalized translation operator and a generalized convolution operator with this transform.
\newline
\indent There are many ways to define the Dunkl translation. We use the definition which most underlines the analogy with the Fourier transform. It is the definition given in \cite{tx} with a different convention. 
\newline 
Let $x \in \mathbb R^d$. The Dunkl translation operator $\tau^\kappa_x$ is given for $f \in L^2(\mu_\kappa)$ by
\[
\mathcal F_\kappa\bigl(\tau^\kappa_x(f)\bigr)(y)=E_\kappa(ix,y)\mathcal F_\kappa(f)(y), \quad y \in \mathbb R^d.
\]
It plays the role of $f\mapsto f(\cdot+x)$ in Fourier analysis. It is important to note that it is not a positive operator. The following explicit formula for $\tau^\kappa_x$ is due to Rösler~(see~\cite{ro1}). In the case $G\simeq\mathbb Z_2$, we have for a continuous function $f$ on $\mathbb R$ and for $x, y \in \mathbb R$
\begin{multline} \label{explicite}
\tau^\kappa_x(f)(y)=\frac{1}{2}\int_{-1}^1f\Bigl(\sqrt{x^2+y^2+2xyt}\Bigr)\biggl(1+\frac{x+y}{\sqrt{x^2+y^2+2xyt}}\biggr)\Phi_\kappa(t)\,\mathrm{d}t
\\ +\frac{1}{2}\int_{-1}^1f\Bigl(-\sqrt{x^2+y^2+2xyt}\Bigr)\biggl(1-\frac{x+y}{\sqrt{x^2+y^2+2xyt}}\biggr)\Phi_\kappa(t)\,\mathrm{d}t,
\end{multline}
where $\Phi_\kappa(t)=M_\kappa(1+t)(1-t^2)^{\kappa-1}$. It follows from \eqref{explicite} a formula for $\tau^\kappa_x$ in the case $G\simeq\mathbb Z_2^d$ and this formula implies the boundedness of $\tau^\kappa_x$ (it is still a challenging problem for a general reflection group).
\begin{pr} \label{bornitude}
Let $x \in \mathbb R^d$. The operator $\tau^\kappa_x$ extends to $L^p(\mu_\kappa)$ for $p \in [1,+\infty]$ and for $f \in L^p(\mu_\kappa)$ we have
\[
\bigl\|\tau^\kappa_x(f)\bigr\|_{\kappa,p}\leqslant C\|f\|_{\kappa,p},
\]
where $C$ is independent of $x$ and $f$.
\end{pr}
The last result we mention about the generalized translation is the following one-dimensional inequality which has been recently proved by C. Abdelkefi and M.~Sifi in \cite{as} (see also \cite{bxu}).
\begin{pr} \label{abdelkefisifi}
There exists a positive constant $C$ such that for $x, y \in \mathbb R$ and for every $r>0$ we have
\[
\bigl|\tau^\kappa_x(\chi_{_{[-r,r]}})(y)\bigr|\leqslant C\,\frac{\mu_\kappa\bigl(]-r,r[\bigr)}{\mu_\kappa\bigl(I(x,r)\bigr)},
\]
where we denote by $I(x,r)$ the following set
\[
I(x,r)=\bigl[\max\{0;|x|-r\},|x|+r\bigr[.
\]
\end{pr}
We conclude this section with the definition and the basic properties of  the Dunkl convolution operator. According to \cite{tx}, this operator is defined for both $f$ and $g$ in $L^2(\mu_\kappa)$ by
\[
(f*_\kappa g)(x)=c_\kappa\int_{\mathbb R^d}f(y)\tau^\kappa_x(g)(-y)\,\mathrm d\mu_\kappa(y), \quad x \in \mathbb R^d.
\] 
Thanks to Proposition \ref{bornitude}, the usual Young's inequality holds (for the proof, see for instance \cite{zy}).
\begin{pr}
Assume that $p^{-1}+q^{-1}=1+r^{-1}$ with $p,q,r \in [1,+\infty]$. Then, the map $(f,g)\mapsto f*_\kappa g$ defined on $L^2(\mu_\kappa)\times L^2(\mu_\kappa)$ extends to a continuous map from $L^p(\mu_\kappa)\times L^q(\mu_\kappa)$ to $L^r(\mu_\kappa)$ and we have 
\[
\|f*_\kappa g\|_{\kappa,r}\leqslant C\|f\|_{\kappa,p}\,\|g\|_{\kappa,q},
\]
where $C$ is independent of  $f$ and $g$.
\end{pr}
We finally note that the Dunkl convolution satisfies the  properties $f*_\kappa g=g*_\kappa f$ and $\mathcal F_\kappa(f~{*_\kappa g)=\mathcal F_\kappa(f)\cdot\mathcal F_\kappa(g)}$.
\section{Fefferman-Stein inequalities}
This section is concerned with the proof of our Fefferman-Stein inequalities, that is to say Theorem~\ref{mainresult}. In fact, as we have already claimed, the proof is straightforward once we have constructed an operator $M_\kappa^R$ which controls $M_\kappa$ and which satisfies the classical Fefferman-Stein inequalities. What we have in mind for the construction of  $M_\kappa^R$ is that we want to use the sharp inequality of Proposition~\ref{abdelkefisifi} because it is a key argument to bypass the lack of information on the Dunkl translation operator. Nevertheless, this proposition is one-dimensional. This is the reason for which we shall introduce a Dunkl maximal operator $M_\kappa^Q$ defined with cubes. Indeed, the basic observation  $\chi_{{}_{\overline{Q}_r}}(x)=\prod_{j=1}^d\chi_{{}_{[-r,r]}}(x_j)$ (together with the fact that~$E_\kappa(x,y)=\prod_{j=1}^dE_{\kappa_j}(x_j,y_j)$) will allow us to prove the formula
\[
\tau^\kappa_x(\chi_{{}_{{\overline{Q}}_r}})(y)=\prod_{j=1}^d\tau^{\kappa_j}_{x_j}(\chi_{{}_{[-r,r]}})(y_j),
\]
from which we will deduce not only the definition of the operator $M_\kappa^R$ but also the inequality $M_\kappa^Qf\leqslant M_\kappa^Rf$.
Therefore, in order to prove the inequality \eqref{eq1}, it will be enough to prove that  $M_\kappa^Q$ controls $M_\kappa$. Since $\tau^\kappa_x$ is not a positive operator, it is not at all obvious that they are connected. Thus, we shall study how they are related to each other.
\newline
\indent First of all, we introduce the auxiliary operator $M_\kappa^Q$.
\begin{de}
Let $M^Q_\kappa$ be the Dunkl maximal operator defined with cubes centered at the origin and whose sides are parallel to the axes by 
\[
M^Q_\kappa f(x)=\sup_{r>0}\frac{1}{\mu_\kappa\bigl(Q_r\bigr)}\biggl|\int_{\mathbb R^d}f(y)\tau^\kappa_x(\chi_{{}_{Q_r}})(-y)\,\mathrm{d}\mu_\kappa(y)\biggr|, \quad x \in \mathbb R^d, 
\]
where for every $r>0$ we set $Q_r=\bigl\{x \in \mathbb R^d: |x_j|<r,\ j=1,\ldots,d\bigr\}$. 
\end{de}
Our first aim is to prove that this maximal operator controls $M_\kappa$. In view of this, we need the following lemma. Before stating it, we have to introduce a notation.
\begin{nota}
For $x, y \in \mathbb R\setminus\{0\}$, we denote by $\nu^{\kappa,+}_{x,y}$ the measure given for every $z \in \mathbb R$ by
\[
\mathrm{d}\nu^{\kappa,+}_{x,y}(z)=\frac{1}{2}\,K_{\kappa}\bigl(|x|,|y|,|z|\bigr)(1-\sigma_{x,y,z})\,\mathrm{d}\mu_\kappa(z).
\]
\end{nota}
Let us point out that this measure is positive. Indeed, it is a simple consequence of the following observation
\[
|z| \in \Bigl[\bigl||x|-|y|\bigr|,|x|+|y|\Bigr]\Longrightarrow |\sigma_{x,y,z}|\leqslant 1.
\]
With this notation in mind, we can now formulate the lemma.
\begin{lem} \label{cle}
Let $x=(x_1,\ldots,x_d) \in \mathbb R_{\,\mathrm{reg}}^d$. Then $\tau^\kappa_x(\chi_{{}_{{\overline{Q}}_r}})$ is a positive function on~$\mathbb R_{\,\mathrm{reg}}^d$ and for $y=(y_1,\ldots,y_d) \in \mathbb R_{\,\mathrm{reg}}^d$ we have
\[
\tau^\kappa_x(\chi_{{}_{{\overline{Q}}_r}})(y)=\int_{\mathbb R^d}\chi_{{}_{{\overline{Q}}_r}}(z)\,\mathrm{d}\upsilon_{x,y}(z),
\]
where the measure $\upsilon^\kappa_{x,y}$ is given by
\[
\mathrm{d}\upsilon^\kappa_{x,y}(z)=\mathrm{d}\nu^{\kappa_1,+}_{x_1,y_1}(z_1)\cdots\mathrm{d}\nu^{\kappa_d,+}_{x_d,y_d}(z_d).
\]
\end{lem}
Before we come to the proof of this lemma, let us introduce the so-called Dunkl heat kernel $q_\kappa^t$ which is associated with the Dunkl Laplacian $\Delta_\kappa=\sum_{j=1}^d\mathcal D_j^2$. This kernel is given for every $t>0$ by
\[
q_\kappa^t(\cdot)=\frac{1}{(2t)^{\gamma_\kappa+\frac{d}{2}}}\,\mathrm{e}^{-\frac{\|\cdot\|^2}{4t}}.
\]
It satisfies $\mathcal F_\kappa(q_\kappa^t)(\cdot)=\mathrm{e}^{-t\|\cdot\|^2}$ and the following equality
\begin{equation} \label{tdhk}
\tau^\kappa_x(q_\kappa^t)(y)=\frac{1}{(2t)^{\gamma_\kappa+\frac{d}{2}}}\,\mathrm{e}^{-\frac{\|x\|^2+\|y\|^2}{4t}}\,E_\kappa\biggl(\frac{x}{\sqrt{2t}},-\frac{y}{\sqrt{2t}}\biggr), \quad x,y \in \mathbb R^d.
\end{equation}
Moreover, we know that $\tau^\kappa_x(q_\kappa^t)(y)>0$ for  $x$ and $y$ in $\mathbb R^d$ and that
\begin{equation} \label{idhk}
\int_{\mathbb R^d}\tau^\kappa_{x}(q_\kappa^t)(y)\,\mathrm{d}\mu_\kappa(y)=\frac{1}{c_\kappa}.
\end{equation}
For all these results (and for more details), the reader may consult \cite{ro3} or \cite{ro5}.
\newline
\indent We now turn to the proof of Lemma \ref{cle}.
\begin{proof}
One begins with the proof of the following  one-dimensional equality
\begin{equation}
\label{txchir} \tau^\kappa_x(\chi_{{}_{[-r,r]}})(y)=\int_{\mathbb{R}}\chi_{{}_{[-r,r]}}(z)\,\mathrm{d}\nu^{\kappa,+}_{x,y}(z), \quad x, y \in \mathbb R\setminus\{0\}.
\end{equation}
\indent Let $q_\kappa^t$ be the Dunkl heat kernel defined above. 
\newline
We readily observe that $\chi_{{}_{[-r,r]}}*_\kappa q_\kappa^t \in L^1(\mu_\kappa)$, which implies, on account of Proposition~\ref{bornitude}, that $\tau^\kappa_x(\chi_{{}_{[-r,r]}}*_\kappa q_\kappa^t) \in L^1(\mu_\kappa)$. Moreover,  we have by Hölder's inequality and Plancherel's theorem
\[
\bigl\|\mathcal F_\kappa(\chi_{{}_{[-r,r]}})\cdot\mathcal F_\kappa(q_\kappa^t)\bigr\|_{\kappa,1}\leqslant \bigl\|\chi_{{}_{[-r,r]}}\bigr\|_{\kappa,2}\bigl\|q_\kappa^t\bigr\|_{\kappa,2},
\] 
from which we deduce that
\[
\mathcal F_\kappa(\chi_{{}_{[-r,r]}}*_\kappa q_\kappa^t)=\mathcal F_\kappa(\chi_{{}_{[-r,r]}})\cdot\mathcal F_\kappa(q_\kappa^t) \in L^1(\mu_\kappa).
\]
Since we have by definition
\[
\mathcal F_\kappa\bigl(\tau^\kappa_x(\chi_{{}_{[-r,r]}}*_\kappa q_\kappa^t)\bigr)(\cdot)=E_\kappa(ix,\cdot)\mathcal F_\kappa(\chi_{{}_{[-r,r]}}*_\kappa q_\kappa^t)(\cdot),
\]
then $\mathcal F_\kappa\bigl(\tau^\kappa_x(\chi_{{}_{[-r,r]}}*_\kappa q_\kappa^t)\bigr) \in L^1(\mu_\kappa)$ and we can apply the inversion formula to obtain
\[
\tau^\kappa_x(\chi_{{}_{[-r,r]}}*_\kappa q_\kappa^t)(y)=c_\kappa\int_\mathbb{R} E_\kappa(ix,z)E_\kappa(iy,z)\mathcal F_\kappa(\chi_{{}_{[-r,r]}})(z)\mathrm{e}^{-tz^2}\,\mathrm{d}\mu_\kappa(z).
\]
If we now use the product formula of Proposition \ref{productformula} we get
\begin{align*}
\tau^\kappa_x(\chi_{{}_{[-r,r]}}*_\kappa q_\kappa^t)(y)&=c_\kappa\int_\mathbb{R}\biggl(\int_\mathbb{R} E_\kappa(iz,z')\,\mathrm{d}\nu^\kappa_{x,y}(z')\biggr)\mathcal F_\kappa(\chi_{{}_{[-r,r]}})(z)\mathrm{e}^{-tz^2}\,\mathrm{d}\mu_\kappa(z)\\
&=c_\kappa\int_\mathbb{R}\biggl(\int_\mathbb{R} E_\kappa(iz,z')\mathcal F_\kappa(\chi_{{}_{[-r,r]}})(z)\mathrm{e}^{-tz^2}\,\mathrm{d}\mu_\kappa(z)\biggr)\mathrm{d}\nu^\kappa_{x,y}(z'),
\end{align*}
from which we deduce thanks to the inversion formula
\begin{equation} \label{partiel}
\tau^\kappa_x(\chi_{{}_{[-r,r]}}*_\kappa q_\kappa^t)(y)=\int_\mathbb{R} (\chi_{{}_{[-r,r]}}*_\kappa q_\kappa^t)(z')\,\mathrm{d}\nu^\kappa_{x,y}(z').
\end{equation}
But we claim that $\chi_{{}_{[-r,r]}}*_\kappa q_\kappa^t$ is an even function. Indeed
\begin{align*}
(\chi_{{}_{[-r,r]}}*_\kappa q_\kappa^t)(-\xi)&=c_\kappa\int_{\mathbb R}\chi_{{}_{[-r,r]}}(\xi')\tau^\kappa_{-\xi}(q_\kappa^t)(-\xi')\,\mathrm{d}\mu_\kappa(\xi')\\
&=c_\kappa\int_{\mathbb R}\chi_{{}_{[-r,r]}}(\xi')\tau^\kappa_{\xi}(q_\kappa^t)(\xi')\,\mathrm{d}\mu_\kappa(\xi')=(\chi_{{}_{[-r,r]}}*_\kappa q_\kappa^t)(\xi),
\end{align*}
where we have used the definition of $*_\kappa$ in the first step, the formula \eqref{tdhk} in the second step (in order to prove that $\tau^\kappa_{-\xi}(q_\kappa^t)(-\xi')=\tau^\kappa_{\xi}(q_\kappa^t)(\xi')$) and a change of variables and the definition of the Dunkl convolution in the last step. 
\newline
Since  both $z\mapsto \sigma_{z,x,y}$ and $z\mapsto \sigma_{z,y,x}$ are odd functions, the equality \eqref{partiel} is therefore equivalent to the following one
\begin{equation} \label{txchir1}
\tau^\kappa_x(\chi_{{}_{[-r,r]}}*_\kappa q_\kappa^t)(y)=\int_\mathbb{R} (\chi_{{}_{[-r,r]}}*_\kappa q_\kappa^t)(z')\,\mathrm{d}\nu^{\kappa,+}_{x,y}(z').
\end{equation} 
In order to prove \eqref{txchir} we will take limit in \eqref{txchir1} as $t$ goes to $0$.
Observe that, by Plancherel's theorem
\begin{align*}
\bigl\|\chi_{{}_{[-r,r]}}*_\kappa q_\kappa^t-\chi_{{}_{[-r,r]}}\bigr\|^2_{\kappa,2}&=\bigl\|\mathcal F_\kappa(\chi_{{}_{[-r,r]}})\cdot\mathcal F_\kappa(q_\kappa^t)-\mathcal F_\kappa(\chi_{{}_{[-r,r]}})\bigr\|^2_{\kappa,2}\\
&=\int_{\mathbb{R}}\bigl|\mathcal F_\kappa(\chi_{{}_{[-r,r]}})(\xi)\bigr|^2\bigl(1-\mathrm{e}^{-t\xi^2}\bigr)^2\,\mathrm{d}\mu_\kappa(\xi).
\end{align*}
Thus, $\chi_{{}_{[-r,r]}}*_\kappa q_\kappa^t \to \chi_{{}_{[-r,r]}}$ in $L^2(\mu_\kappa)$ as $t \to 0$. Since $\tau^\kappa_x$ is a bounded operator on~$L^2(\mu_\kappa)$ we also  have $\tau^\kappa_x(\chi_{{}_{[-r,r]}}*_\kappa q_\kappa^t) \to \tau^\kappa_x(\chi_{{}_{[-r,r]}})$ in $L^2(\mu_\kappa)$ as $t \to 0$. By passing to a subsequence if necessary we can therefore assume that the convergence is also almost everywhere. Taking limit as $t$ goes to $0$ in \eqref{txchir1} gives us 
\[
\tau^\kappa_x(\chi_{{}_{[-r,r]}})(y)=\lim_{t\to0}\int_\mathbb{R} (\chi_{{}_{[-r,r]}}*_\kappa q_\kappa^t)(z')\,\mathrm{d}\nu^{\kappa,+}_{x,y}(z').
\]
Then \eqref{txchir} is proved if we show the following equality
\begin{equation}
\label{txchir2} \lim_{t\to0}\int_\mathbb{R} (\chi_{{}_{[-r,r]}}*_\kappa q_\kappa^t)(z')\,\mathrm{d}\nu^{\kappa,+}_{x,y}(z')=\int_{\mathbb{R}}\chi_{{}_{[-r,r]}}(z')\,\mathrm{d}\nu^{\kappa,+}_{x,y}(z').
\end{equation}
In view of this, we shall use the Lebesgue dominated convergence theorem. Since the almost everywhere convergence of $\chi_{{}_{[-r,r]}}*_\kappa q_\kappa^t$ to $\chi_{{}_{[-r,r]}}$ has been already proved above, it suffices to majorize $|\chi_{{}_{[-r,r]}}*_\kappa q_\kappa^t|$ by a function independent of $t$ and which is integrable with respect to $\nu^{\kappa,+}_{x,y}$. 
\newline
By the definition of the Dunkl convolution
\[
(\chi_{{}_{[-r,r]}}*_\kappa q_\kappa^t)(z')=c_\kappa\int_{\mathbb{R}}\chi_{{}_{[-r,r]}}(\xi)\tau^\kappa_{z'}(q_\kappa^t)(-\xi)\,\mathrm{d}\mu_\kappa(\xi),
\]
from which we deduce that
\[
\bigl|(\chi_{{}_{[-r,r]}}*_\kappa q_\kappa^t)(z')\bigr|\leqslant c_\kappa\int_{\mathbb{R}}\bigl|\tau^\kappa_{z'}(q_\kappa^t)(-\xi)\bigr|\,\mathrm{d}\mu_\kappa(\xi)=c_\kappa\int_{\mathbb{R}}\tau^\kappa_{z'}(q_\kappa^t)(\xi)\,\mathrm{d}\mu_\kappa(\xi),
\]
where we have used  the positivity of $\tau^\kappa_{z'}(q_\kappa^t)$ and a change of variables in the last step.
\newline
On account of \eqref{idhk} we then obtain
\[
\bigl|(\chi_{{}_{[-r,r]}}*_\kappa q_\kappa^t)(z')\bigr|\leqslant 1.
\]
Since the function equal to $1$ is integrable with respect to $\nu^{\kappa,+}_{x,y}$, the Lebesgue dominated convergence theorem allows us to complete the proof of \eqref{txchir2} and then \eqref{txchir} is proved. 
\newline
Let us point out that we deduce from \eqref{txchir} the positivity of $\tau^\kappa_x(\chi_{{}_{[-r,r]}})$.
\newline
\indent We next prove the following equality
\begin{equation} \label{taucube}
\tau^\kappa_x(\chi_{{}_{{\overline{Q}}_r}})(y)=\prod_{j=1}^d\tau^{\kappa_j}_{x_j}(\chi_{{}_{[-r,r]}})(y_j), \quad x,y \in \mathbb R_{\,\text{reg}}^d.
\end{equation}
We can apply the inversion formula (by a reprise of the argument given above) to obtain
\begin{equation} \label{invtaucube}
\tau^\kappa_x(\chi_{{}_{{\overline{Q}}_r}}*_\kappa q_\kappa^t)(y)=c_\kappa\int_{\mathbb R^d}E_\kappa(ix,z)E_\kappa(iy,z)\mathcal F_\kappa(\chi_{{}_{{\overline{Q}}_r}})(z)\mathrm{e}^{-t\|z\|^2}\,\mathrm{d}\mu_\kappa(z).
\end{equation}
Let us notice that we have the following product formula
\begin{equation} \label{dunklcube}
\mathcal F_\kappa(\chi_{{}_{{\overline{Q}}_r}})(z)=\prod_{j=1}^d\mathcal F_{\kappa_j}(\chi_{{}_{[-r,r]}})(z_j), \quad z \in \mathbb R^d.
\end{equation}
Indeed, by the definition of the Dunkl transform we have
\[
\mathcal F_\kappa(\chi_{{}_{{\overline{Q}}_r}})(z)=c_\kappa\int_{\mathbb R^d} E_\kappa(z,-iz')\chi_{{}_{{\overline{Q}}_r}}(z')\,\mathrm{d}\mu_\kappa(z').
\]
Since we can separate the variables we get
\[
\mathcal F_\kappa(\chi_{{}_{{\overline{Q}}_r}})(z)=\prod_{j=1}^d\biggl(\int_{\mathbb R}c_{\kappa_j}E_{\kappa_j}(z_j,-iz'_j)\chi_{{}_{[-r,r]}}(z'_j)h^2_{\kappa_j}(z'_j)\,\mathrm{d} z'_j\biggr),
\]
from which \eqref{dunklcube} follows. We combine \eqref{dunklcube} with \eqref{invtaucube} to obtain
\begin{multline*}
\tau^\kappa_x(\chi_{{}_{{\overline{Q}}_r}}*_\kappa q_\kappa^t)(y)\\
=\prod_{j=1}^d\biggl(\int_{\mathbb R}c_{\kappa_j}E_{\kappa_j}(ix_j,z_j)E_{\kappa_j}(iy_j,z_j)\mathcal F_{\kappa_j}(\chi_{{}_{[-r,r]}})(z_j)\mathrm{e}^{-tz_j^2}h^2_{\kappa_j}(z_j)\,\mathrm{d} z_j\biggr),
\end{multline*}
that is to say
\[
\tau^\kappa_x(\chi_{{}_{{\overline{Q}}_r}}*_\kappa q_\kappa^t)(y)=\prod_{j=1}^d\tau^{\kappa_j}_{x_j}(\chi_{{}_{[-r,r]}}*_{\kappa_j} q_{\kappa_j}^t)(y_j),
\]
from which we deduce \eqref{taucube} by taking limit.
\newline
The proof of the lemma is now obvious. Indeed, using the equality \eqref{txchir} in \eqref{taucube} gives us
\[
\tau^\kappa_x(\chi_{{}_{{\overline{Q}}_r}})(y)=\prod_{j=1}^d\int_{\mathbb{R}}\chi_{{}_{[-r,r]}}(z_j)\,\mathrm{d}\nu^{\kappa_j,+}_{x_j,y_j}(z_j),
\]
which is precisely what we wanted to prove.
\end{proof} 
We are now in a position to prove that $M_\kappa^Q$ controls $M_\kappa$. More precisely, we have the following proposition.
\begin{pr} \label{besoin} 
There exists a positive constant $C$ such that for every $x \in \mathbb R^d_{\,\mathrm{reg}}$ we have 
\[
0\leqslant M_\kappa f(x)\leqslant CM^Q_\kappa |f|(x).
\]
\end{pr}

\begin{proof} 
Thanks to the definition of $M_\kappa$ there is nothing to do for the first inequality. 
\newline \indent We now turn to the second one.\newline
Let $x \in \mathbb R^d_{\,\mathrm{reg}}$ and $r>0$. Let us remark that we readily have
\begin{equation} \label{int1}
\int_{\mathbb R^d}f(y)\tau^\kappa_x(\chi_{{}_{{B}_r}})(-y)\,\mathrm{d}\mu_\kappa(y)=\int_{\mathbb R^d_{\,\text{reg}}}f(y)\tau^\kappa_x(\chi_{{}_{\overline{{B}}_r}})(-y)\,\mathrm{d}\mu_\kappa(y).
\end{equation}
\newline
The key argument for the proof is that we can show, even if $\tau^\kappa_x$ is not a positive operator, the following inequality
\begin{equation} \label{int2}
0\leqslant\tau^\kappa_x(\chi_{{}_{{\overline{B}}_r}})(y)\leqslant\tau^\kappa_x(\chi_{{}_{{\overline{Q}}_r}})(y), \quad x, y \in \mathbb R^d_{\,\mathrm{reg}}.
\end{equation}
Thanks to the explicit formula of $\tau^\kappa_x(\chi_{{}_{{\overline{Q}}_r}})$ given in the previous lemma, it is enough to show that 
\begin{equation} \label{repboule}
\tau^\kappa_x(\chi_{{}_{{\overline{B}}_r}})(y)=\int_{\mathbb R^d}\chi_{{}_{{\overline{B}}_r}}(z)\,\mathrm{d}\upsilon^\kappa_{x,y}(z), \quad x,y \in \mathbb R^d_{\,\mathrm{reg}},
\end{equation}
in order to prove \eqref{int2}. Therefore, we now turn to the proof of \eqref{repboule}. By a reprise of the argument given in the proof of Lemma \ref{cle}, we can apply the inversion formula to write for both $x$ and $y$ in $\mathbb R^d_{\,\mathrm{reg}}$
\[
\tau^\kappa_x(\chi_{{}_{{\overline{B}}_r}}*_\kappa q_\kappa^t)(y)=c_\kappa\int_{\mathbb{R}^d} E_\kappa(ix,z)E_\kappa(iy,z)\mathcal F_\kappa(\chi_{{}_{{\overline{B}}_r}})(z)\mathrm{e}^{-t\|z\|^2}\,\mathrm{d}\mu_\kappa(z).
\]
Since $E_\kappa(x,y)=\prod_{j=1}^dE_{\kappa_j}(x_j,y_j)$, we have thanks to Proposition \ref{productformula}
\begin{multline*}
\tau^\kappa_x(\chi_{{}_{{\overline{B}}_r}}*_\kappa q_\kappa^t)(y)\\
=c_\kappa\int_{\mathbb{R}^d}\biggl(\int_{\mathbb{R}^d}E_\kappa(iz,z')\,\mathrm{d}\nu^{\kappa_1}_{x_1,y_1}(z'_1)\cdots\mathrm{d}\nu^{\kappa_d}_{x_d,y_d}(z'_d)\biggr)\mathcal F_\kappa(\chi_{{}_{{\overline{B}}_r}})(z)\mathrm{e}^{-t\|z\|^2}\,\mathrm{d}\mu_\kappa(z),
\end{multline*}
from which it follows
\begin{multline*}
\tau^\kappa_x(\chi_{{}_{{\overline{B}}_r}}*_\kappa q_\kappa^t)(y)\\
=c_\kappa\int_{\mathbb{R}^d}\biggl(\int_{\mathbb{R}^d}E_\kappa(iz,z')\mathcal F_\kappa(\chi_{{}_{{\overline{B}}_r}})(z)\mathrm{e}^{-t\|z\|^2}\,\mathrm{d}\mu_\kappa(z)\biggr)\,\mathrm{d}\nu^{\kappa_1}_{x_1,y_1}(z'_1)\cdots\mathrm{d}\nu^{\kappa_d}_{x_d,y_d}(z'_d).
\end{multline*}
We apply the inversion formula to get
\[
\tau^\kappa_x(\chi_{{}_{{\overline{B}}_r}}*_\kappa q_\kappa^t)(y)=\int_{\mathbb{R}^d}(\chi_{{}_{{\overline{B}}_r}}*_\kappa q_\kappa^t)(z')\,\mathrm{d}\nu^{\kappa_1}_{x_1,y_1}(z'_1)\cdots\mathrm{d}\nu^{\kappa_d}_{x_d,y_d}(z'_d),
\]
and we obtain thanks to the Fubini theorem
\begin{equation} \label{iter}
\tau^\kappa_x(\chi_{{}_{{\overline{B}}_r}}*_\kappa q_\kappa^t)(y)=\int_{\mathbb{R}^{d-1}}\biggl(\int_{\mathbb R}(\chi_{{}_{{\overline{B}}_r}}*_\kappa q_\kappa^t)(z')\,\mathrm{d}\nu^{\kappa_1}_{x_1,y_1}(z'_1)\biggr)\,\mathrm{d}\nu^{\kappa_2}_{x_2,y_2}(z'_2)\cdots\mathrm{d}\nu^{\kappa_d}_{x_d,y_d}(z'_d).
\end{equation}
Since $\chi_{{}_{{\overline{B}}_r}}$ is radial, $\chi_{{}_{{\overline{B}}_r}}*_\kappa q_\kappa^t$ is also radial. Therefore, it is even with respect to each of its variables, that is to say $(\chi_{{}_{{\overline{B}}_r}}*_\kappa q_\kappa^t)(\varepsilon_1z_1,\ldots,\varepsilon_dz_d)=(\chi_{{}_{{\overline{B}}_r}}*_\kappa q_\kappa^t)(z_1,\ldots,z_d)$ with $\varepsilon_j=\pm1$. Then \eqref{iter} is equivalent to
\[
\tau^\kappa_x(\chi_{{}_{{\overline{B}}_r}}*_\kappa q_\kappa^t)(y)=\int_{\mathbb{R}^{d-1}}\biggl(\int_{\mathbb R}(\chi_{{}_{{\overline{B}}_r}}*_\kappa q_\kappa^t)(z')\,\mathrm{d}\nu^{\kappa_1,+}_{x_1,y_1}(z'_1)\biggr)\,\mathrm{d}\nu^{\kappa_2}_{x_2,y_2}(z'_2)\cdots\mathrm{d}\nu^{\kappa_d}_{x_d,y_d}(z'_d).
\]
\newline
By successive uses of the Fubini theorem  we are readily led to
\begin{equation} \label{fina}
\tau^\kappa_x(\chi_{{}_{{\overline{B}}_r}}*_\kappa q_\kappa^t)(y)=\int_{\mathbb{R}^d}(\chi_{{}_{{\overline{B}}_r}}*_\kappa q_\kappa^t)(z')\,\mathrm{d}\nu^{\kappa_1,+}_{x_1,y_1}(z'_1)\cdots\mathrm{d}\nu^{\kappa_d,+}_{x_d,y_d}(z'_d).
\end{equation}
Taking limit as $t$ tends to $0$ in \eqref{fina} gives us \eqref{repboule} which in turn implies \eqref{int2}.
\newline
Consequently, if we apply \eqref{int2} in \eqref{int1} we are led to the following inequality
\[
\biggl|\int_{\mathbb R^d}f(y)\tau^\kappa_x(\chi_{{}_{{B}_r}})(-y)\,\mathrm{d}\mu_\kappa(y)\biggr|
\leqslant\int_{\mathbb R^d_{\,\text{reg}}}|f(y)|\tau^\kappa_x(\chi_{{}_{{\overline{Q}}_r}})(-y)\,\mathrm{d}\mu_\kappa(y).
\] 
Since it is obvious that
\[
\int_{\mathbb R^d_{\,\text{reg}}}|f(y)|\tau^\kappa_x(\chi_{{}_{{\overline{Q}}_r}})(-y)\,\mathrm{d}\mu_\kappa(y)=\int_{\mathbb R^d}|f(y)|\tau^\kappa_x(\chi_{{}_{{Q}_r}})(-y)\,\mathrm{d}\mu_\kappa(y),
\]
we can therefore write
\[
\biggl|\int_{\mathbb R^d}f(y)\tau^\kappa_x(\chi_{{}_{{B}_r}})(-y)\,\mathrm{d}\mu_\kappa(y)\biggr|
\leqslant\int_{\mathbb R^d}|f(y)|\tau^\kappa_x(\chi_{{}_{{Q}_r}})(-y)\,\mathrm{d}\mu_\kappa(y),
\]
from which it follows at once that
\begin{equation} \label{interm}
\frac{1}{\mu_\kappa(B_r)}\biggl|\int_{\mathbb R^d}f(y)\tau^\kappa_x(\chi_{{}_{{B}_r}})(-y)\,\mathrm{d}\mu_\kappa(y)\biggr|\leqslant\frac{1}{\mu_\kappa(B_r)}\int_{\mathbb R^d}|f(y)|\tau^\kappa_x(\chi_{{}_{{Q}_r}})(-y)\,\mathrm{d}\mu_\kappa(y).
\end{equation}
Let us notice that $\mu_\kappa(Q_r)=C\mu_\kappa(B_r)$ with
\[
C=\frac{2^d(2\gamma_\kappa+d)}{\prod_{j=1}^d\bigl(2\kappa_j+1\bigr)}\biggl(\int_{S^{d-1}}h_\kappa^2(y)\,\mathrm{d} y\biggr)^{-1}.
\]
Indeed, we have on one hand
\[
\mu_\kappa(Q_r)=\prod_{j=1}^d\mu_{\kappa_j}\bigl(]-r,r[\bigr)=2^d\prod_{j=1}^d\Bigl(\frac{1}{2\kappa_j+1}\Bigr)r^{2\gamma_\kappa+d},
\]
and on the other hand, changing to polar coordinates gives 
\[
\mu_\kappa(B_r)=\int_0^r u^{2\gamma_\kappa+d-1}\,\mathrm{d} u\int_{S^{d-1}}h_\kappa^2(y)\,\mathrm{d} y=\frac{1}{2\gamma_\kappa+d}\biggl(\int_{S^{d-1}}h_\kappa^2(y)\,\mathrm{d} y\biggr)r^{2\gamma+d},
\]
where we have used the fact that $h_\kappa^2$ is homogeneous of degree $2\gamma_\kappa$. \newline We can therefore reformulate \eqref{interm} as follows
\[
\frac{1}{\mu_\kappa(B_r)}\biggl|\int_{\mathbb R^d}f(y)\tau^\kappa_x(\chi_{{}_{{B}_r}})(-y)\,\mathrm{d}\mu_\kappa(y)\biggr|\leqslant \frac{C}{\mu_\kappa(Q_r)}\int_{\mathbb R^d}|f(y)|\tau^\kappa_x(\chi_{{}_{{Q}_r}})(-y)\,\mathrm{d}\mu_\kappa(y),
\]
from which we deduce that
\[
\frac{1}{\mu_\kappa(B_r)}\biggl|\int_{\mathbb R^d}f(y)\tau^\kappa_x(\chi_{{}_{{B}_r}})(-y)\,\mathrm{d}\mu_\kappa(y)\biggr|\leqslant CM^Q_\kappa |f|(x),
\]
and then the result.
\end{proof}

Thanks to this proposition, it is enough to construct an operator $M_\kappa^R$ which controls $M_\kappa^Q$ in order to prove the inequality \eqref{eq1}. Before we come to the definition of $M_\kappa^R$ we give some notations.
\begin{notas} For $z=(z_1,\ldots,z_d) \in \mathbb R^d$ we put $\tilde{z}=\bigl(|z_1|,\ldots,|z_d|\bigr)$ and we denote by~$R(z,r)$ (for every $r>0$) the following set
\[
R(z,r)=I(z_1,r)\times\cdots\times I(z_d,r).
\]
Recall that we have defined for $x \in \mathbb R$ and $r>0$ the set $I(x,r)$ by
\[
I(x,r)=\bigl[\max\{0;|x|-r\},|x|+r\bigr[.
\]
\end{notas}

Since we want to use the sharp inequality of Poposition \ref{abdelkefisifi} together with the fact that $\tau^\kappa_x(\chi_{{}_{{\overline{Q}}_r}})(y)=\prod_{j=1}^d\tau_{x_j}^{\kappa_j}(\chi_{{}_{[-r,r]}})(y_j)$, we are naturally led to introduce the following operator.
\begin{de} Let $M^R_\kappa$ be the weighted maximal operator defined by
\[
M^R_\kappa f(x)=\sup_{r>0}\frac{1}{\mu_\kappa\bigl(R(x,r)\bigr)}\int_{\tilde{y} \in R(x,r)}|f(y)|\,\mathrm{d}\mu_\kappa(y), \quad x \in \mathbb R^d.
\]
\end{de}

This operator satisfies the classical properties of maximal operators. Let us clarify our statement.
\newline
Since $\mu_\kappa$ is a doubling weight, we have the following covering  lemma (a one-dimensional result for $I(x,r)$ can be found in \cite{as} or \cite{bxu}).
\begin{lem}
Let $E$ be a measurable (with respect to $\mu_\kappa$) subset of $\mathbb R_+~{\times\cdots\times\mathbb R_+}$. Suppose $E\subset\cup_{j \in J}R_j$ with $R_j=R(z_j,r_j)$ bounded for every $j \in J$ (where $z_j \in \mathbb R^d$ and $r_j>0$).  Then, from this family, we can choose a sequence (which may be finite) of disjoint sets $R_1,\ldots,R_n,\ldots$, such that 
\[
\mu_\kappa(E)\leqslant C\sum_n\mu_\kappa(R_n),
\] 
where $C$ is a positive constant which depends only on $\kappa_1,\ldots,\kappa_d$.
\end{lem}
Thanks to this lemma, a weak-type $(1,1)$ result  for $M_\kappa^R$ can be easily proved. Indeed, if we set 
\[
E_+=\Bigl\{x \in \mathbb R_+^*\times\cdots\times\mathbb R_+^*:M_\kappa^Rf(x)>\lambda\Bigr\},
\]
we can choose (thanks to the definition of $M_\kappa^R$ and the covering lemma)  a suitable sequence of disjoint sets $R_n$ such that $\mu_\kappa(E_+)\leqslant C\sum_n\mu_\kappa(R_n)$, where~$C$ depends only on $\kappa_1,\ldots,\kappa_d$. We can then follow the standard techniques~(see for instance~\cite{stt}) in order to prove that $\mu_\kappa(E_+)\leqslant\tfrac{C}{\lambda}\|f\|_{\kappa,1}$. 
\newline Finally, the basic but crucial observation 
\begin{equation} \label{crucial}
M^R_\kappa f(x)=M_\kappa^Rf(\varepsilon_1x_1,\ldots,\varepsilon_dx_d),
\end{equation}
with $\varepsilon_j=\pm1$, allows us to deduce the weak-type inequality, that is
\[
\mu_\kappa\Bigl(\Bigl\{x \in \mathbb R^d:M_\kappa^Rf(x)>\lambda\Bigr\}\Bigr)\leqslant\frac{C}{\lambda}\|f\|_{\kappa,1}.
\]
Since $M^R_\kappa$ is obviously bounded on $L^\infty$, the weak-type $(1,1)$ inequality implies the strong-type $(p,p)$ inequality by the Marcinkiewicz interpolation theorem (see \cite{stt}). Thus, we have proved the following maximal theorem for $M_\kappa^R$.
\begin{thm} \label{maxim}
Let $f$ be a function defined on $\mathbb R^d$.
\begin{enumerate}
\item If $f \in L^1(\mu_\kappa)$, then  for every $\lambda>0$ we have
\[
\mu_\kappa\Bigl(\Bigl\{x \in \mathbb R^d: M_\kappa^Rf(x)>\lambda\Bigr\}\Bigr)\leqslant\frac{C}{\lambda}\|f\|_{\kappa,1},
\]
where $C$ is a positive constant independent of $f$ and $\lambda$.
\item If $f \in L^p(\mu_\kappa)$, $1<p\leqslant+\infty$, then $M_\kappa^Rf \in L^p(\mu_\kappa)$ and we have
\[
\|M_\kappa^Rf\|_{\kappa,p}\leqslant C\|f\|_{\kappa,p},
\]
where $C$ is a positive constant independent of $f$.
\end{enumerate}
\end{thm}
Moreover, we claim that the following weighted inequality is true.
\begin{lem}
Let $W$ be a positive and locally integrable (with respect to $\mu_\kappa$) function defined on $\mathbb R^d$. For $1<q<+\infty$, there exists a positive constant $C$ which depends only on $\kappa_1,\ldots,\kappa_d$ and $q$ and  such that
\[
\int_{\mathbb R^d}\bigl(M_\kappa^Rf(y)\bigr)^qW(y)\,\mathrm{d}\mu_\kappa(y)\leqslant C\int_{\mathbb R^d}|f(y)|^qM_\kappa^RW(y)\,\mathrm{d}\mu_\kappa(y).
\]
\end{lem}
Indeed, by the Marcinkiewicz interpolation theorem, this lemma is an immediate consequence of the trivial fact that $M_\kappa^R$ is bounded on $L^\infty$ together with the following inequality
\begin{equation} \label{pourpoids}
\tilde{\mu}_\kappa\Bigl(\Bigl\{x \in \mathbb R^d: M^R_\kappa f(x)>\lambda\Bigr\}\Bigr)\leqslant\frac{C}{\lambda}\int_{\mathbb R^d}|f(y)|M^R_\kappa W(y)\,\mathrm{d}\mu_\kappa(y),
\end{equation}
where $\tilde{\mu}_\kappa(X)=\int_XW(y)\,\mathrm{d}\mu_\kappa(y)$ and where $C$ is a positive constant which depends only on $\kappa_1,\ldots,\kappa_d$. The just-written inequality is easy to prove. Indeed, we can show the key inequality
\[
\tilde{\mu}_\kappa(K)\leqslant\frac{C}{\lambda}\int_{\mathbb R^d}|f(y)|M^R_\kappa W(y)\,\mathrm{d}\mu_\kappa(y)
\]
for any compact set $K$ in $E_+$ just as in the proof for the classical maximal operator~(see \cite{stgros}). Therefore
\[
\tilde{\mu}_\kappa(E_+)\leqslant\frac{C}{\lambda}\int_{\mathbb R^d}|f(y)|M^R_\kappa W(y)\,\mathrm{d}\mu_\kappa(y)
\]
and we then deduce \eqref{pourpoids} on account of \eqref{crucial}.
\newline
\indent To conclude, we claim that we can combine the maximal theorem and the weighted inequality for $M_\kappa^R$ with a Calder\'on-Zygmund decomposition of $f$~(see for instance \cite{stt}) to obtain the Fefferman-Stein inequalities for $M_\kappa^R$ following almost verbatim the proof in \cite{fs}.

\begin{thm} \label{fsmkappa}
Let  $(f_n)_{n\geqslant1}$ be a sequence of measurable functions defined on $\mathbb R^d$. 
\begin{enumerate} 
\item If $1<r<+\infty$, $1<p<+\infty$ and if $\bigl(\sum_{n=1}^\infty|f_n(\cdot)|^r\bigr)^\frac{1}{r} \in L^p(\mu_\kappa)$, then we have 
\[
\biggl\|\Bigl(\sum_{n=1}^\infty|M^R_\kappa f_n(\cdot)|^r\Bigr)^\frac{1}{r}\biggr\|_{\kappa,p}\leqslant C\biggl\|\Bigl(\sum_{n=1}^\infty|f_n(\cdot)|^r\Bigr)^\frac{1}{r}\biggr\|_{\kappa,p},
\]
where $C=C(\kappa_1,\ldots,\kappa_d,r,p)$ is independent of $(f_n)_{n\geqslant1}$.
\item If $1<r<+\infty$ and if  $\bigl(\sum_{n=1}^\infty|f_n(\cdot)|^r\bigr)^\frac{1}{r} \in L^1(\mu_\kappa)$, then for every $\lambda>0$ we have 
\[
\mu_\kappa\biggl(\biggl\{x \in \mathbb R^d: \Bigl(\sum_{n=1}^\infty|M^R_\kappa f_n(x)|^r\Bigr)^\frac{1}{r}>\lambda\biggr\} \biggr)\leqslant \frac{C}{\lambda}\biggl\|\Bigl(\sum_{n=1}^\infty|f_n(\cdot)|^r\Bigr)^\frac{1}{r}\biggr\|_{\kappa,1},
\]
where $C=C(\kappa_1,\ldots,\kappa_d,r)$ is independent of $(f_n)_{n\geqslant1}$ and $\lambda$.
\end{enumerate}
\end{thm}

Therefore, in order to prove Theorem \ref{mainresult}, it remains to show that the operator~$M_\kappa^R$ controls $M_\kappa^Q$. More precisely, we have the following proposition.

\begin{pr} \label{besoin2}
There exists a positive constant $C$ such that for every $x \in \mathbb R_{\,\mathrm{reg}}^d$ we have 
\[
M^Q_\kappa f(x)\leqslant CM^R_\kappa f(x).
\]
\end{pr}

\begin{proof}Let $x \in \mathbb R_{\,\mathrm{reg}}^d$ and $r>0$. By the definition of the Dunkl convolution we have
\[
\bigl|(f*_\kappa\chi_{{}_{{Q}_r}})(x)\bigr|=c_\kappa\biggl|\int_{\mathbb R^d}f(y)\tau^\kappa_x(\chi_{{}_{{Q}_r}})(-y)\,\mathrm{d}\mu_\kappa(y)\biggr|,
\]
from which we deduce at once that
\[
\bigl|(f*_\kappa\chi_{{}_{{Q}_r}})(x)\bigr|=c_\kappa\biggl|\int_{\mathbb R_{\,\text{reg}}^d}f(y)\tau^\kappa_x(\chi_{{}_{{\overline{Q}}_r}})(-y)\,\mathrm{d}\mu_\kappa(y)\biggr|.
\]
Using the positivity of $\tau^\kappa_x(\chi_{{}_{{\overline{Q}}_r}})$ gives us
\[
\bigl|(f*_\kappa\chi_{{}_{{Q}_r}})(x)\bigr|\leqslant c_\kappa\int_{\mathbb R_{\,\text{reg}}^d}|f(y)|\tau^\kappa_x(\chi_{{}_{{\overline{Q}}_r}})(-y)\,\mathrm{d}\mu_\kappa(y).
\]
On account of \eqref{taucube} we then obtain
\[
\bigl|(f*_\kappa\chi_{{}_{{Q}_r}})(x)\bigr|\leqslant c_\kappa\int_{\mathbb R_{\,\text{reg}}^d}|f(y)|\prod_{j=1}^d\tau^{\kappa_j}_{x_j}(\chi_{{}_{[-r,r]}})(-y_j)\,\mathrm{d}\mu_\kappa(y).
\]
Since we can readily deduce from \eqref{txchir} the following  property 
\[
|y_j| \notin I(x_j,r)\Longrightarrow\tau^{\kappa_j}_{x_j}(\chi_{{}_{[-r,r]}})(y_j)=0,
\]
we can write
\[
\bigl|(f*_\kappa\chi_{{}_{{Q}_r}})(x)\bigr|\leqslant c_\kappa\int_{A_x}|f(y)|\prod_{j=1}^d\tau^{\kappa_j}_{x_j}(\chi_{{}_{[-r,r]}})(-y_j)\,\mathrm{d}\mu_\kappa(y),
\]
where $A_x$ is the following set
\[
A_x=\mathbb R^d_{\,\mathrm{reg}}\cap\bigl\{y \in \mathbb R^d: \tilde{y} \in R(x,r)\bigr\}.
\] 
If we now apply the inequality of Proposition \ref{abdelkefisifi} we get
\[
\bigl|(f*_\kappa\chi_{{}_{{Q}_r}})(x)\bigr|\leqslant C\int_{A_x}|f(y)|\prod_{j=1}^d\frac{\mu_{\kappa_j}\bigl(]-r,r[\bigr)}{\mu_{\kappa_j}\bigl(I(x_j,r)\bigr)}\,\mathrm{d}\mu_\kappa(y).
\]
The following obvious equalities
\[
\prod_{j=1}^d\mu_{\kappa_j}\bigl(]-r,r[\bigr)=\mu_\kappa(Q_r),\quad \prod_{j=1}^d\mu_{\kappa_j}\bigl(I(x_j,r)\bigr)=\mu_\kappa\bigl(R(x,r)\bigr),
\]
imply that
\[
\bigl|(f*_\kappa\chi_{{}_{{Q}_r}})(x)\bigr|\leqslant \frac{C\mu_\kappa(Q_r)}{\mu_\kappa\bigl(R(x,r)\bigr)}\int_{A_x}|f(y)|\,\mathrm{d}\mu_\kappa(y),
\]
from which we deduce that 
\[
\frac{1}{\mu_\kappa(Q_r)}\bigl|(f*_\kappa\chi_{{}_{{Q}_r}})(x)\bigr|\leqslant \frac{C}{\mu_\kappa\bigl(R(x,r)\bigr)}\int_{\tilde{y} \in R(x,r)}|f(y)|\,\mathrm{d}\mu_\kappa(y).
\]
It follows that 
\[
\frac{1}{\mu_\kappa(Q_r)}\bigl|(f*_\kappa\chi_{{}_{{Q}_r}})(x)\bigr|\leqslant CM_\kappa^Rf(x),
\]
and then the result.
\end{proof}

This result, combined with Proposition \ref{besoin},  leads immediately to the following corollary.
\begin{co} \label{corol}
There exists a positive constant $C$ such that for every $x \in \mathbb R^d_{\,\mathrm{reg}}$ we have 
\[
0\leqslant M_\kappa f(x)\leqslant CM^R_\kappa f(x).
\]
\end{co}

Then, Theorem \ref{mainresult} is true thanks to this corollary  and the Fefferman-Stein inequalities for $M_\kappa^R$ (Theorem \ref{fsmkappa}).

\begin{rem}
Let us point out that Corollary \ref{corol}, together with the maximal result for~$M_\kappa^R$ (Theorem \ref{maxim}), implies a maximal theorem for $M_\kappa$ (proved in \cite{tx}) without using the Hopf-Dunford-Schwartz ergodic theorem (which is a general method given in  \cite{stlp}). 
\end{rem}

\section{Application}

Since the Fefferman-Stein inequalities are an important tool in Harmonic analysis, we would like to define a large class of operators such that each operator of this class satisfies these inequalities, and such that, in particular, the maximal operator associated with the Dunkl heat semigroup and the maximal operator associated with the Dunkl-Poisson semigroup belong to this class (see \cite{stlp} for details about the classical heat semigroup and the classical Poisson semigroup).\newline
\indent To become more precise, let us now introduce this class of operators.
\begin{de}
Let $\phi \in L^1(\mu_\kappa)$ be a radial function, that is $\phi(x)=\tilde{\phi}\bigl(\|x\|\bigr)$ for every~{$x \in \mathbb R^d$}, such that $\tilde{\phi}$  is differentiable and satisfies the following properties
\[
\lim_{r\to\infty}\tilde{\phi}(r)=0, \quad \int_0^\infty r^{2\gamma_\kappa+d}\Bigl|\frac{\mathrm{d}}{\mathrm{d} r}\tilde\phi(r)\Bigr|\,\mathrm{d} r<+\infty.
\]
Then we denote by $M^\phi_\kappa$ the following operator
\[
M^\phi_\kappa f(x)=\sup_{t>0}\bigl|(f*_\kappa \phi_t)(x)\bigr|, \quad x \in \mathbb R^d,
\]
where  $\phi_t$ is for every $t>0$  the dilation of $\phi$ given by
\[
\phi_t(x)=\frac{1}{t^{2\gamma_\kappa+d}}\,\phi\Bigl(\frac{x}{t}\Bigr), \quad x \in \mathbb R^d.
\] 
\end{de}

Let us present two important examples of functions which satisfy the conditions of the previous definition.
\newline The first one is concerned with the Dunkl heat kernel $q_\kappa^t$. Indeed if  we let  
\[
\phi(x)=\mathrm{e}^{-\frac{\|x\|^2}{2}}, \quad x \in \mathbb R^d,
\]
then for every $t>0$ we have
\[
\phi_{{\sqrt{2t}}}(x)=\frac{1}{(2t)^{\gamma_\kappa+\frac{d}{2}}}\,\mathrm{e}^{-\frac{\|x\|^2}{4t}}=q_\kappa^t(x).
\]
In this case, $M_\kappa^\phi$ is therefore the maximal function of the Dunkl heat semigroup.
Our second example deals with the Dunkl-Poisson kernel. If we define the function~$\phi$ for every $x \in \mathbb R^d$ by
\[
\phi(x)=\frac{a_\kappa}{(1+\|x\|^2)^{\gamma_\kappa+\frac{d+1}{2}}}, \ \ \text{with}\ \ a_\kappa=\frac{c_\kappa\,2^{\gamma_\kappa+\frac{d}{2}}\,\Gamma\bigl(\gamma_\kappa+\frac{d+1}{2}\bigr)}{\sqrt{\pi}},
\]
then for every $t>0$ we have 
\[
\phi_t(x)=\frac{a_\kappa \,t}{(t^2+\|x\|^2)^{\gamma_\kappa+\frac{d+1}{2}}}=P_\kappa^t(x),
\]
which is the Dunkl-Poisson kernel (for more details about this kernel, the reader is referred to \cite{rovoit} and \cite{tx}). Thus, in this case, $M_\kappa^\phi$ is the maximal function associated with the Dunkl-Poisson semigroup. 
\newline \indent We now state the Fefferman-Stein inequalities for $M_\kappa^\phi$ (for $\phi, \tilde{\phi}$ and $\phi_t$ as above).
\begin{thm}
Let  $(f_n)_{n\geqslant1}$ be a sequence of measurable functions defined on $\mathbb R^d$. 
\begin{enumerate} 
\item If $1<r<+\infty$, $1<p<+\infty$ and if $\bigl(\sum_{n=1}^\infty|f_n(\cdot)|^r\bigr)^\frac{1}{r} \in L^p(\mu_\kappa)$, then we have 
\[
\biggl\|\Bigl(\sum_{n=1}^\infty|M^\phi_\kappa f_n(\cdot)|^r\Bigr)^\frac{1}{r}\biggr\|_{\kappa,p}\leqslant C\biggl\|\Bigl(\sum_{n=1}^\infty|f_n(\cdot)|^r\Bigr)^\frac{1}{r}\biggr\|_{\kappa,p},
\]
where $C=C(\phi,\kappa_1,\ldots,\kappa_d,r,p)$ is independent of $(f_n)_{n\geqslant1}$.
\item If $1<r<+\infty$ and if  $\bigl(\sum_{n=1}^\infty|f_n(\cdot)|^r\bigr)^\frac{1}{r} \in L^1(\mu_\kappa)$, then for every $\lambda>0$ we have
\[
\mu_\kappa\biggl(\biggl\{x \in \mathbb R^d: \Bigl(\sum_{n=1}^\infty|M^\phi_\kappa f_n(x)|^r\Bigr)^\frac{1}{r}>\lambda\biggr\} \biggr)\leqslant \frac{C}{\lambda}\biggl\|\Bigl(\sum_{n=1}^\infty|f_n(\cdot)|^r\Bigr)^\frac{1}{r}\biggr\|_{\kappa,1},
\]
where $C=C(\phi,\kappa_1,\ldots,\kappa_d,r)$ is independent of $(f_n)_{n\geqslant1}$ and $\lambda$.
\end{enumerate}
\end{thm}

\begin{proof}
The proof is nearly obvious. Indeed, according to the proof of Theorem $7.5$ in \cite{tx}, we have for such a function $\phi$ and for $x \in \mathbb R^d$ 
\[
\bigl|(f*_\kappa\phi)(x)\bigr|\leqslant CM_\kappa f(x)\int_0^\infty r^{2\gamma_\kappa+d}\Bigl|\frac{\mathrm{d}}{\mathrm{d} r}\tilde\phi(r)\Bigr|\,\mathrm{d} r,
\]
where $C$ depends only on $\kappa_1,\ldots,\kappa_d$. Therefore, for every $t>0$ we get  
\[
\bigl|(f*_\kappa\phi_t)(x)\bigr|\leqslant CM_\kappa f(x)\int_0^\infty r^{2\gamma_\kappa+d}\Bigl|\frac{\mathrm{d}}{\mathrm{d} r}\widetilde{\phi_t}(r)\Bigr|\,\mathrm{d} r,
\]
with $C$ independent of $t$. Since we have 
\[
\frac{\mathrm{d}}{\mathrm{d} r}\widetilde{\phi_t}(r)=\frac{1}{t^{2\gamma_\kappa+d+1}}\frac{\mathrm{d}}{\mathrm{d} r}\tilde{\phi}\Bigl(\frac{r}{t}\Bigr),
\]
we can write
\[
\bigl|(f*_\kappa\phi_t)(x)\bigr|\leqslant CM_\kappa f(x)\int_0^\infty \frac{r^{2\gamma_\kappa+d}}{t^{2\gamma_\kappa+d+1}}\Bigl|\frac{\mathrm{d}}{\mathrm{d} r}\tilde{\phi}\Bigl(\frac{r}{t}\Bigr)\Bigr|\,\mathrm{d} r.
\]
A change of variables gives us
\[
\bigl|(f*_\kappa\phi_t)(x)\bigr|\leqslant CM_\kappa f(x)\int_0^\infty r^{2\gamma_\kappa+d}\Bigl|\frac{\mathrm{d}}{\mathrm{d} r}\tilde\phi(r)\Bigr|\,\mathrm{d} r,
\]
from which we deduce that
\[
\sup_{t>0}\bigl|(f*_\kappa\phi_t)(x)\bigr|\leqslant CM_\kappa f(x),
\]
where $C$ depends only on $\kappa_1,\ldots,\kappa_d$ and $\phi$. If we now apply Theorem \ref{mainresult} we obtain the desired result.
\end{proof}
\bibliography{dimdv2revised}

\begin{thebibliography}{10}

\bibitem{as}
Chokri Abdelkefi and Mohamed Sifi.
\newblock Dunkl translation and uncentered maximal operator on the real line.
\newblock {\em Int. J. Math. Math. Sci.}, pages Art. ID 87808, 9, 2007.

\bibitem{bxu}
Walter~R. Bloom and Zeng~Fu Xu.
\newblock The {H}ardy-{L}ittlewood maximal function for {C}h\'ebli-{T}rim\`eche
  hypergroups.
\newblock In {\em Applications of hypergroups and related measure algebras
  (Seattle, WA, 1993)}, volume 183 of {\em Contemp. Math.}, pages 45--70. Amer.
  Math. Soc., Providence, RI, 1995.

\bibitem{dJ}
M.~F.~E. de~Jeu.
\newblock The {D}unkl transform.
\newblock {\em Invent. Math.}, 113:147--162, 1993.

\bibitem{d3}
Charles~F. Dunkl.
\newblock Differential-difference operators associated to reflection groups.
\newblock {\em Trans. Amer. Math. Soc.}, 311:167--183, 1989.

\bibitem{d2}
Charles~F. Dunkl.
\newblock Hankel transforms associated to finite reflection groups.
\newblock In {\em Hypergeometric functions on domains of positivity, Jack
  polynomials, and applications (Tampa, FL, 1991)}, volume 138 of {\em Contemp.
  Math.}, pages 123--138. Amer. Math. Soc., Providence, RI, 1992.

\bibitem{fs}
C.~Fefferman and E.~M. Stein.
\newblock Some maximal inequalities.
\newblock {\em Amer. J. Math.}, 93:107--115, 1971.

\bibitem{ro1}
Margit R{\"o}sler.
\newblock Bessel-type signed hypergroups on {${\bf R}$}.
\newblock In {\em Probability measures on groups and related structures, XI
  (Oberwolfach, 1994)}, pages 292--304. World Sci. Publ., River Edge, NJ, 1995.

\bibitem{ro3}
Margit R{\"o}sler.
\newblock Generalized {H}ermite polynomials and the heat equation for {D}unkl
  operators.
\newblock {\em Comm. Math. Phys.}, 192:519--542, 1998.

\bibitem{ro2}
Margit R{\"o}sler.
\newblock Positivity of {D}unkl's intertwining operator.
\newblock {\em Duke Math. J.}, 98:445--463, 1999.

\bibitem{ro5}
Margit R{\"o}sler.
\newblock Dunkl operators: theory and applications.
\newblock In {\em Orthogonal polynomials and special functions ({L}euven,
  2002)}, volume 1817 of {\em Lecture Notes in Math.}, pages 93--135. Springer,
  Berlin, 2003.

\bibitem{ro4}
Margit R{\"o}sler.
\newblock A positive radial product formula for the {D}unkl kernel.
\newblock {\em Trans. Amer. Math. Soc.}, 355:2413--2438, 2003.

\bibitem{rovoit}
Margit R{\"o}sler and Michael Voit.
\newblock Markov processes related with {D}unkl operators.
\newblock {\em Adv. in Appl. Math.}, 21:575--643, 1998.

\bibitem{stt}
Elias~M. Stein.
\newblock {\em Singular integrals and differentiability properties of
  functions}.
\newblock Princeton University Press, Princeton, N.J., 1970.

\bibitem{stlp}
Elias~M. Stein.
\newblock {\em Topics in harmonic analysis related to the {L}ittlewood-{P}aley
  theory.}
\newblock Princeton University Press, Princeton, N.J., 1970.

\bibitem{stgros}
Elias~M. Stein.
\newblock {\em Harmonic analysis: real-variable methods, orthogonality, and
  oscillatory integrals}.
\newblock Princeton University Press, Princeton, NJ, 1993.

\bibitem{tx}
Sundaram Thangavelu and Yuan Xu.
\newblock Convolution operator and maximal function for the {D}unkl transform.
\newblock {\em J. Anal. Math.}, 97:25--55, 2005.

\bibitem{tx2}
Sundaram Thangavelu and Yuan Xu.
\newblock Riesz transform and {R}iesz potentials for {D}unkl transform.
\newblock {\em J. Comput. Appl. Math.}, 199:181--195, 2007.

\bibitem{tr}
Khalifa Trim{\`e}che.
\newblock Paley-{W}iener theorems for the {D}unkl transform and {D}unkl
  translation operators.
\newblock {\em Integral Transforms Spec. Funct.}, 13:17--38, 2002.

\bibitem{wa}
G.~N. Watson.
\newblock {\em A {T}reatise on the {T}heory of {B}essel {F}unctions}.
\newblock Cambridge University Press, Cambridge, England, 1944.

\bibitem{xuin}
Yuan Xu.
\newblock Orthogonal polynomials for a family of product weight functions on
  the spheres.
\newblock {\em Canad. J. Math.}, 49:175--192, 1997.

\bibitem{zy}
A.~Zygmund.
\newblock {\em Trigonometric series. 2nd ed. {V}ols. {I}, {II}}.
\newblock Cambridge University Press, New York, 1959.

\end{thebibliography}
\end{document}